\documentclass[12pt]{article}

\usepackage{amsmath, amsthm, amsfonts}
\usepackage{listings}
\usepackage{amsfonts}
\usepackage{amssymb}
\usepackage{float}
\usepackage{anysize}
\usepackage{graphicx}
\usepackage{mathrsfs} 
\usepackage{hyperref}
\usepackage{cite}
\usepackage{color}
\usepackage{epstopdf}

\newtheorem{thm}{Theorem}[section]

\newtheorem{lem}[thm]{Lemma}
\newtheorem{prop}[thm]{Proposition}
\theoremstyle{definition}

\theoremstyle{remark}
\newtheorem{rem}[thm]{Remark}

\title{Statistical Inference in Fractional Poisson Ornstein-Uhlenbeck Process.}
\author{H\'ector Araya\thanks{CIMFAV, Facultad de Ingenier\'ia, Universidad de Valpara\'iso. Email:  hector.araya@postgrado.uv.cl} \and Natalia Bahamonde\thanks{Instituto de Estad\'istica, Pontificia Universidad Cat\'olica de Valpara\'iso, Chile. Email: natalia.bahamonde@pucv.cl.} \and Tania Roa \thanks{CIMFAV, Facultad de Ingenier\'ia, Universidad de Valpara\'iso. Email: tania.roa@postgrado.uv.cl} \and Soledad Torres\thanks{CIMFAV, Facultad de Ingenier\'ia, Universidad de Valpara\'iso. Email: soledad.torres@uv.cl} 
}

\begin{document}
\maketitle

\abstract{

In this article, we study the problem of parameter estimation for a discrete Ornstein - Uhlenbeck model driven by Poisson fractional noise. Based on random walk approximation for the noise, we study least squares and maximum likelihood estimators. Thus, asymptotic behaviours of the estimator is carried out, and a simulation study is shown to illustrate our results.}\\

\section{Introduction}
The fractional Poisson process introduced by Laskin in 2003 \cite{Las} is a counting, non-Gaussian long-memory process constructed as a fractional non-Markov Poisson stochastic process, based on fractional generalization of the Kolmogorov-Feller equation. In, \cite{Las}, the authors compute the probability of $n$ arrivals by time $t$, exhibits the long-memory effect, and also compute the waiting time distribution as a non-exponential density related to the Mittag-Leffler function.

At the same time, Wang et al \cite{Wan1} defined the fractional Poisson process (fPp) as  a class of non-Gaussian processes with stationary increments, by using the integral with respect to the same kernel as the fractional Brownian motion, but replacing Brownian motion by a compensated Poisson process with intensity $\lambda >0$. 

This construction yields that almost surely all paths of the process are continuous and of H\"older index strictly less than $H-1/2$, for $ H> 1/2$, and share many properties in common with the classical Gaussian fractional Brownian motion (fBm): it exhibits  long-range dependence, it has the same covariance structure, has wide-sense stationary increments, among others.


From the non-Gaussianity of the fPp, we find evidence of the heavy tails in the underlying process $N_t$, while the fBm does not capture this property. There are other few options for defining processes with heavier tails, which also exhibits long memory and a weak version of the self-similarity property. A special class of non Gaussian long memory self-similar processes, that are limits in the so called Non Central Limit Theorem, is the Rosenblatt process, studied for example in \cite{cro}, \cite{tud} and \cite{tv}, and the references therein. In the case of long memory, with heavier tails process we can mention  the  t-Student process, which has the property that at each instant $t$, has a t-Student distribution with $t$ degrees of freedom. For more references  see \cite{PILAR} . 

Thus, the use of Poisson noise, which we advocate in this article, is an alternative way of capturing heavy tails, and it provides for the property of long-range dependence in a more natural way than for the Rosenblatt process, since it uses a kernel integral definition based on a process with independent increments. In other words, the long memory constructed in fPp goes deeper than in the Rosenblatt process, since in that case, long memory is just a property of the process's covariance structure. The Poisson noise also helps leave the realm of constructions based on normal or second-chaos fluctuations. 

In this work we consider the problem of estimate the parameter $\theta$ in the following linear stochastic differential equation driven by  fractional Poisson process
\begin{equation}\label{intro1}
dX_{t} = \theta X_{t} \, dt + dN^{H}_{t} , \ \ t \in [0,T],
\end{equation}

We recall here, the two well known approaches to estimate the parameter in a linear fractional Gaussian Ornstein - Uhlenbeck model. For a more general picture of the topic see for example \cite{kes1}. 

\begin{itemize}
\item Maximum Likelihood Estimation (MLE):  In general, the use maximum likelihood techniques relies on knowing the density function explicitly in order to be able to perform the necessary analysis to maximize the score function. For fBM models, this technique is an application of the Girsanov theorem, and depends on the properties  of deterministic fractional operators  determined by $H$, the Hurst parameter, (see \cite{klep} , \cite{rao} and  \cite{viens1}). Also, there are MLE methods based on numerical approximations for the fractional linear models, in this case we can mention the work of \cite{ber2} , \cite{ber1} and \cite{Rif}, where  the authors approximates the model by the Euler scheme approximation, and the fBm by a disturbed random walk \cite{Sot}.

\item Least Square Estimation (LSE): This method has been studied in \cite{hu1}, where the authors used malliavin calculus to prove the consistency of the estimator. On the other hand, in the ergodic case, the study of consistency for LSE has been treated for many authors \cite{azm1}, \cite{cenac} , \cite{onsy}, \cite{azm2}, \cite{hu1}, \cite{hu2}. In the non-ergodic case we can mention the works: \cite{kes3}, \cite{mendy}, \cite{kes1}, \cite{kes2}, where the authors use the H\"older regularity of the fractional Gaussian process and malliavin calculus to prove the consistency of the estimator and their asymptotic distribution.
\end{itemize}

The problem of estimating stochastic parameters in a model driven by fBm starting in the 2000s, but the case of non Gaussian process with long memory is relatively new, since in this case, is not clear how to get an explicit likelihood function. We can mention here the  work given in \cite{ber1}, where the authors deals the Rossenblatt process, which one is a non Gaussian long memory process.


From here, is where it came the motivation of our work, we want to take a step in this direction and work in the framework of estimation in a non-Gaussian long memory model. That is why, in this article, we will consider a simple model given by the solution $X$ of a linear stochastic differential equation given in (\ref{intro1}),  where $ \left( N^{H}_{t} \right)_{t \in [0,T]}$ is a fractional Poisson process with Hurst index $H>1/2$, and $\theta \in \mathbb{R}$ is an unknown parameter. Our plan is to estimate $\theta$ using discrete time observations of $X$, driven by an approximate model of equation (\ref{intro1}). Thus, we will work under the context of \cite{ber2}, \cite{ber1} and \cite{Rif}. We will mention some of the main difficulties in this article.

\begin{itemize}
\item The finite dimensional distributions of the process fPp are not known, and of course are not Gaussian.
\item The weak convergence theorems (Fourth moment theorem) related to Poisson integrals greatly differs from the Gaussian context. \cite{PD}. 
\item The H\"older regularity of the process fPp is less than the regularity of the fBm. 
\item It is a big task manipulate Wiener-It\^o integrals in the  Poisson framework, than the Gaussian one. 
\end{itemize}

Since the  process $X$ given by the unique solution of equation (\ref{intro1}) has an autoregressive structure, we can prove the conditional mean square consistency for the LSE, and in the case of the  MLE  estimator, we are able to construct the estimator by means of a transformation of the noise which includes a Bernoulli random variables.  This is an extra task in this issue. In order to study the rate of convergence and the  asymptotic distribution,  we use Monte Carlo simulation, besides some techniques of approximation for integrals. This allows us, to show for different values of $H$, the shape of the limits distribution. We mention here that the main problem to obtain the asymptotic distribution is related to the autoregressive structure of the  model.

 Furthermore, we define an extra parameter $\alpha$ that control the number of samples ensuring  the convergence results. In the case of the model related to fractional Poisson process we need that $\alpha >3/2$ (see Theorem (\ref{theorem1}) and Theorem (\ref{theorem2})), on the other hand when there is a relation between $\alpha$ and $\lambda$. In particular, if we take $\lambda = m \ln(2)$, then we recover the model related to fractional Brownian motion and the necessary condition in this case is $\alpha >1 $. The difference between the fPp model and fBm model is expected and is due to the structure of the different random walks related to the  approximation of both process. In view of the aforementioned, it is that working in the context of articles \cite{ber2}, \cite{ber1} and \cite{Rif} give us a chance to solve the problem of the consistency of the estimators.

Finally we organized the rest of our paper as follows. In section 2, we give some preliminaries related to fractional Poisson process and the fractional Poisson random walk. In Section 3, the model is presented. Section 4 is devoted to the behaviour analysis of the least square estimator. Section 5  considers the maximum likelihood estimator. Finally a simulation study and a discussion is reported in Section 6.


\section{Preliminaries and main results}
This section introduce the basic notions that we will need throughout the paper. First, we introduce the fractional Poisson process and some elements related to the approximation of the fPp, afterwards the model is presented. Finally,  we  present our main results.  

\subsection{Fractional Poisson process.}
Let $N_{t}^{H}$ a fractional Poisson process with  $\left( H \ \in \ \left( \frac{1}{2}, 1  \right) \right)$. The stochastic process $N_{t}^{H}$ with $t \in [0,T]$ defined by 
\begin{equation}
N_{t}^{H}= \frac{1}{\Gamma \left( H-\frac{1}{2} \right)} \int_{0}^{t}u^{\frac{1}{2}-H}{\left(  \int_{u}^{t} \tau^{H-\frac{1}{2}} \left( \tau - u \right)^{H-\frac{3}{2}} d\tau \right)} dq(u),
\end{equation}
where $q(u)= \frac{N(u)}{\sqrt{\lambda}}-\sqrt{\lambda}u$  and $N(u)$ is a homogeneous Poisson process with intensity  $\lambda \geq 0$. Following \cite{Wan1}  we have: 
\begin{itemize}
\item[(i)] The covariance  function of $N_{t}^{H}$ is given by 
\begin{equation}
 \hspace{-1.1cm} \mathbb{E}(N_{t}^{H}N_{s}^{H})= \frac{V_H^2}{2}\left( t ^{2H} +    s ^{2H} - \vert t-s \vert^{2H} \right) \hspace{0.1cm} \mbox{where}  \hspace{0.1cm} V_{H}^{2}=-\frac{\Gamma(2-2H)cos(\pi H)}{(2H-1) \pi H}
\end{equation}
\item[(ii)] $N_{t}^{H}$ is wide sense self-similar, i.e 
 $\mathbb{E}(N_{t}^{H})=0$, $\mathbb{E}(N_{at}^{H}N_{as}^{H})= a^{2H}\mathbb{E}(N_{t}^{H}N_{s}^{H})$,
and $\mathbb{E}(N_{t}^{H}-N_{s}^{H})^{2}= V_{H}^{2} \vert t-s \vert^{2H}$.
\item[(iii)] The fractional Poisson process $N_{t}^{H}$ has wide sense stationary increments. 
\item[(iv)] The fractional Poisson process $N_{t}^{H}$ exhibits the long range dependence.
\item[(v)] For $H>\frac{1}{2}$ the fractional Poisson process $N_{t}^{H}$ on $[0,T]$ has a.s. H\"{o}lder continuous paths of any order strictly less than $H -\frac{1}{2}$ (see \cite{ar} for details). 
\end{itemize}
It is important to mention that $N^{H}_{t}$ cannot be self-similar in a strict sense. (See theorem 3.12 in  \cite{mish}).
More properties of this process can be found in \cite{Wan1} ,\cite{Wan2}, \cite{mish} and \cite{mish2}. From now we have, 
\begin{equation}\label{Kernel}
K_H(t,s) = {1 \over \Gamma(H-1/2) } s^{\frac{1}{2}-H}{\left(  \int_{s}^{t} \tau^{H-\frac{1}{2}} \left( \tau - s \right)^{H-\frac{3}{2}} d\tau \right)},
\end{equation}
so we have the next representation for $N_t^H$ in the form 
\begin{equation}\label{poisson2}
N_t^H= \int_{0}^{t} K_H(t,s) \, dq(s).
\end{equation}
Let be $\widetilde{N}_{t}= N_{t}-\lambda t$, the compensated Poisson process with intensity $\lambda$. We define $\{{\widetilde N}_k^n: k=0, \ldots n \}$, the non symmetric random walk approximating the compensated Poisson process as:
\begin{equation}\label{eta}
{\widetilde N}_0^n=0; \quad {\widetilde N}_k^n = \sum_{i=1}^k \eta_i^n \quad (k=1, \ldots , n).
\end{equation}
where $\eta_1^n , \ldots , \eta_n^n$ are independent and identically distributed random variables with probabilities, for each k, given by: 
\begin{equation}\label{kappa}
\mathbb{P}(\eta_k^n=\kappa_n-1)= 1- \mathbb{P}(\eta_k^n=\kappa_n)=\kappa_n, \mbox{  where }  \ \kappa_n=e^{\frac{-\lambda}{n}}.
\end{equation}

The following lemma gives the convergence of the Poisson random walk to the centered Poisson process and the proof is given in \cite{Lej}. 


\begin{lem}\label{RWP}
Let  $\displaystyle \widetilde{N}_t^n$  defined as in (\ref{eta}) and (\ref{kappa}). Then $\widetilde{N}_t^n$  is a ${\cal F}^n$- martingale and exists a family $(\phi^n)_{n  \in  \mathbb{N}}$ of one-to-one random time changes from $[0 ,1]$ to $[0, 1]$ such that $\displaystyle \sup_{t \in [0,1]}{\vert  \phi^n(t)-t  \vert} \longrightarrow   0$ almost surely as $n \to \infty$, and $\displaystyle \sup_{t \in [0,1]}{\vert  \widetilde{N}_{t} -\widetilde{N}^n_{\phi^n(t)}\vert} \longrightarrow   0$ in probability. In other words  $\widetilde{N}^{n}$ converges in probability to $\widetilde{N}$ in the $J_{1}$ de Skorokhod topology. 
 \end{lem}

Now, we are ready to define the random walk approximation for Poisson fractional Process. 
Let us define for all $ n \in \mathbb{N}$  the approximation for the Kernel $K$ given in \cite{Sot} by, 
\begin{equation}
K^{n}(t,s):= n \int_{s-\frac{1}{n}}^{s}  {  K\left( \frac{ \lfloor nt \rfloor}{n} , s \right) } ds  \  \mbox{,} n \geq 1  \nonumber
\end{equation}
where $\lfloor x \rfloor$ denotes the greatest integer not exceeding $x$. Following the ideas given in \cite{Sot} and \cite{Torr} we define a discretization of the fractional Poisson Process as: 
\begin{equation}\label{podisc}
N_{t}^{n,H}= {1 \over \sqrt{\lambda}}\int_{0}^{t}{K^{n}(t,s)d\tilde{N}^{n}_{s}}=  {1 \over \sqrt{\lambda}}\sum_{i=1}^{\lfloor nt  \rfloor} {  n \int_{ {i-1 \over n}} ^{i \over n}{K\left( \frac{ \lfloor nt  \rfloor}{n} , s \right) } ds  \times  \eta_{i}^{(n)}}, 
\end{equation}
where  the family or random variables $\eta^{(n)}_i$ are defined as in (\ref{kappa}).

\begin{rem}
In Araya et.al  \cite{ar}, the authors proved that $N^{n,H}$ converges weakly in the Skorohod topology to $N^{H}$ as $n \rightarrow \infty$.
\end{rem} 

\subsection{Parameter estimation in the Ornstein-Uhlenbeck process}
Let us consider the non-Gaussian fractional Ornstein-Uhlenbeck process 
\begin{equation}\label{fou}
X_{t} = \theta \int_{0}^{t} X_{s} ds + N_{t}^{H} , \ \ \ \  t \in [0,T]
\end{equation}
where $N^{H} = \left\lbrace N_{t}^{H}, t \in [0,T] \right\rbrace$ is a fractional Poisson process with
Hurst parameter $H \in (1/2, 1)$, and unknown drift parameter $\theta \in \mathbb{R}$.
Assume that, we have observations $0 = X_{t_{0}} , X_{t_{1}} , \dots, X_{t_{m}}$ and this observations are taken at evenly spaced times $0 = t_{0} < t_{1} < \cdots < t_{m} = 1$. Such observations satisfy the difference equation:
\begin{equation}
X_{j+1} = X_{j} + \theta  X_{j}\Delta t_{j+1} + \Delta N_{j+1}^{H} , \ \ \ 0 \leq j \leq m-1.
\end{equation}
where  $X_{t_{j}} := X_{j} $,  $\Delta t_{j+1} = t_{j+1} -t_{j} = {1 \over m}$,  $\Delta N_{j+1}^{H} = N_{j+1}^{H}  - N_{j}^{H}$ and we also have that $X_{j}$, $N_{j}^{H}$ and  $t_{j}$ depends on $m$. From now, we replace  $N_{t}^{H}$ by $N_{t}^{m,H}$, where $N_{t}^{m,H}$ is given by (\ref{podisc}) and then, we  study the  parameter estimation for the discrete model given by: 
\begin{equation}\label{modeloNN}
X_{j+1} = \left(1 + {\theta \over m} \right) X_{j} + \Delta N_{j+1}^{m,H} , \ \ \ 0 \leq j \leq m-1.
\end{equation}  

Note that for every $j \geq 0$, the increment of the fractional random walk $N^{m,H}$ can be expressed as (see \cite{ber2}):
\begin{equation}\label{incrementoNN}
\Delta N_{j+1}^{m,H} =  \sum_{i=1}^{j} f_{ij} \eta_{i}^{m} + F_{j}\eta_{j+1}^{m},  
\end{equation}

with
\begin{equation}\label{F}
F_{j} :=  m \left( \int_{j \over m }^{j+1 \over m}  K\left( {j+1 \over m} , s \right)  ds \right)
\end{equation}
and 
\begin{equation}\label{f}
f_{ij} :=  m \left( \int_{i-1 \over m }^{i \over m} \left[ K\left( {j+1 \over m} , s \right) - K \left( {j \over m} , s\right) \right] ds \right).
\end{equation}
Let us assume that we have our disposal $m^{\alpha}$ observations, where $\alpha > 1$ will be chosen wisely in order to obtain the convergence result of the estimators. With this in hand and taking into account equations (\ref{modeloNN}) to (\ref{f}), we write:
 
\begin{equation}\label{modeloF}
X_{j+1} = \left(1 + {\theta \over m} \right) X_{j} +  \sum_{i=1}^{j} f_{ij} \eta_{i}^{m^{\alpha}} + F_{j}\eta_{j+1}^{m^{\alpha}} , \ \ \ 0 \leq j \leq m^{\alpha}-1.
\end{equation}

For every $j \geq 1$, the observations $X_{j}$ are non-Gaussian random variables. They are, of course, correlated, and the dependence structure between the observations is rather complicated, involving the elements $f_{i,j}$ and $F_{j}$ previously defined. Moreover it is clear that random variables $\eta_{k}$ depend on the size of the sample. As in \cite{Rif} we will write our model in  (\ref{modeloF}) as:  

\begin{equation}\label{modeloF2}
X_{j+1} = \left(1 + {\theta \over m} \right) X_{j} + T_{j} - {\theta \over m}S_{j-1}  + F_{j} \eta_{j+1}^{m^{\alpha}},  
\end{equation}
where
\begin{equation}\label{T}
T_{j} = T_{j} (x_{1},..., x_{j} ) = \sum_{i=1}^{j} f_{ij} \sum_{k=1}^{i} b_{i,k}^{-1} \cdot x_{k},
\end{equation}
\begin{equation}\label{S}
S_{j-1} = S_{j} (x_{1},..., x_{j-1} ) = \sum_{i=1}^{j} f_{ij} \sum_{k=1}^{i} b_{i,k}^{-1} \cdot (x_{1} + ...+ x_{k-1}).
\end{equation}
 and  $b_{i,k}^{-1} =  \left[ m \int_{i-1 \over m }^{i \over m} K\left( {k \over m } , s \right) ds \right]^{-1}$.

\subsubsection{Least Square Estimator (LSE)}
In this part of the article, we define the LSE for the parameter $\theta$ in the model (\ref{modeloF}). This is achieved using the equivalent form (\ref{modeloF2}), and by formally minimizing: 
\begin{equation*}
\left[  \sum_{j=1}^{m^{\alpha}} {X_{j+1} - \left(1 + {\theta \over m} \right) X_{j} + T_{j} - {\theta \over m}S_{j-1} \over  F_{j} } \right]^{2}
\end{equation*}
with respect to $\theta$, we use the equivalent form (\ref{modeloF2})  since this allows us to simplify the dependence of the model with respect to the error terms $\eta_{i}^{m^{\alpha}}$.  With this in hand, we have that the LSE is given by:
\begin{equation}\label{ls}
\widehat{ \theta }_{LS}^m = m { \sum_{j=1}^{m^{\alpha}} F^{-2}_{j} \cdot \left( \Delta X_{j+1} -  T_{j} \right) \left( X_{j}  - S_{j-1}  \right)  \over \sum_{j=1}^{m^{\alpha}} F^{-2}_{j} \cdot  \left( X_{j}  - S_{j-1} \right)^{2}  }.
\end{equation}

\begin{rem}\label{sigma1}
Is easy to see that the  $\sigma$ - algebra $ \mathscr{F}^{X}_{j}$ generated by $X_{1},..., X_{j}$ coincides with the $\sigma$ -algebra $\mathscr{F}^{\eta}_{j}$ generated by $\eta_{1}^{m^{\alpha}},..., \eta_{j}^{m^{\alpha}}$. 
\end{rem}

By (\ref{modeloF2}) and (\ref{ls}) we get that:
\begin{equation}\label{difls}
\widehat{ \theta }_{Ls}^m - \theta = m { \sum_{j=1}^{m^{\alpha}} F^{-1}_{j} \cdot  \left( X_{j}  - S_{j-1}  \right) \eta_{j+1}^{m^{\alpha}}  \over \sum_{j=1}^{m^{\alpha}} F^{-2}_{j} \cdot \left( X_{j}  - S_{j-1} \right)^{2}  }.
\end{equation}

\subsubsection{Maximum Likelihood Estimator (MLE). }
In this section we give the  MLE  estimator, by means of a transformation of the noise which includes a Bernoulli random variables. 
By using (\ref{T}) and (\ref{S}) we can rewrite  the equation (\ref{modeloF}) as:
\begin{equation}\label{modeloF22}
X_{j+1} = \left(1 + {\theta \over m} \right) X_{j} + T_{j}  - {\theta \over m}S_{j-1} + F_{j} \eta_{j+1}^{m^{\alpha}},  
\end{equation}
here $\eta_{j+1}^{m^{\alpha}}$ is a sequence of  a especial case of Bernoulli random variables, see equation (\ref{kappa}). Now, defining 
\begin{equation}\label{wa1}
B^{m, \alpha}_{j+1} = \eta_{j+1}^{m^{\alpha}} - ( e^{-\lambda / m^{\alpha}} -1 ),
\end{equation}
we obtain that, $B^{m, \alpha}_{\cdot}$ is a sequence of $0-1$ Bernoulli i.i.d  random variables.  Let us recall that from (\ref{wa1}) and \textit{Remark} (\ref{sigma1}), the $\sigma$-algebra $\mathscr{F}^{X}_{j}$ generated by $X_{1}, \dots, X_{j}$ coincides with the $\sigma$-algebra $\mathscr{F}^{{B}_j^{m , \alpha}}$ generated by $B^{m,\alpha}_1,, \ldots   B^{m,\alpha}_j$. 
Since now, we are going to omit superscript  $m , \alpha$ in $  B^{m,\alpha}$. 

From (\ref{wa1}) we can rewrite (\ref{modeloF22}) as
\begin{eqnarray}
X_{j+1} &=& X_{j} + \theta \left(  {X_{j} \over m}    - {S_{j-1} \over m}  \right) + T_{j}  + F_{j} \cdot \left(B_{j+1}+ ( e^{-\lambda / m^{\alpha}} -1 )   \right) \nonumber  \\ 
&=& R_{j}(\theta)  + F_{j} \cdot \left(B_{j+1}+ ( e^{-\lambda / m^{\alpha}} -1 )   \right).  \label{xmodel}
\end{eqnarray}
Here 
\begin{equation}\label{R}
R_{j}(\theta) = X_{j} + \theta \left(  {X_{j} \over m}    - {S_{j-1} \over m}  \right) + T_{j}.
\end{equation}
 Thus, 
\begin{equation}
B_{j+1} = {X_{j+1} - R_{j}(\theta) - F_{j} \cdot (e^{-\lambda / m^{\alpha}} -1)  \over F_{j} }.
\end{equation} 
Now, we are in position to give the expression of the MLE for $\theta$ in (\ref{modeloF}).
\begin{prop}
The maximum likelihood estimator $\hat{\theta}_{ML}^m$ for the parameter $ \theta $ in (\ref{modeloF}) is
given by:
\begin{small}
\begin{equation}\label{ml}
\hat{\theta}_{ML}^m = m {  \sum_{j=1}^{m^{\alpha}} F^{-2}_{j}  \cdot  \left( \Delta X_{j+1} - T_{j} \right)   \left( X_{j} - S_{j-1} \right)  + \left( e^{-\lambda /m^{\alpha}}  -1  \right) \sum_{j=1}^{m^{\alpha}} F^{-1}_{j}  \cdot \left( X_{j} - S_{j-1}  \right)   \over \sum_{j=1}^{m^{\alpha}} F^{-2}_{j} \cdot  \left( X_{j} - S_{j-1} \right)^{2}    }
\end{equation}
\end{small}
with $T_{j}$ , $S_{j-1} $ given by (\ref{T}), (\ref{S}), respectively.
\end{prop}
\begin{proof}
The estimator is based on the conditional law of $X_{j+1}$ given $X_1, \ldots , X_j$. We mention here that we  used $B^{2}_j$ instead $B_j$, since $R_{j}(\theta)$ defined in (\ref{R}) is a linear function of $\theta$, so if calculate the MLE with $B_j$ do not reach a favourable outcome. Also, $B^{2}_{1}, \ldots B^{2}_{n} $, for all $n \in \mathbb{N}$  are i.i.d. random variables with a Bernoulli distribution and parameter $p := p_{m} = e^{-\lambda /m^{\alpha}}$. 

Taking into account that $\mathscr{F}^{X}_{j} = \mathscr{F}^{B}_{j}$ and recalling that if $Z$ is a Bernoulli (0-1) random variable, then $Y=Z^{2}$ is also a Bernoulli (0-1) random variable, and $Y=Z$ a.s. In fact by (\ref{wa1}) and (\ref{xmodel}), we have:

\begin{equation*}
\mathbb{P}\left( B_{j+1} = 0 \right) = \mathbb{P}\left( B_{j+1}^{2} = 0 \right) = \mathbb{P}\left( X_{j+1} = R_{j}(\theta)  + F_{j} ( e^{-\lambda / m^{\alpha}} -1 ) \right) ,   
\end{equation*}
and, since $F_{j} \neq 0$, this implies 
\begin{equation*}
\mathbb{P}\left( B_{j+1}^{2} = 0 \right) = \mathbb{P}\left( X_{j+1} = {R_{j}(\theta)  + F_{j} ( e^{-\lambda / m^{\alpha}} -1 ) \over F_{j}} \right).   
\end{equation*}

Therefore, the likelihood is given by:

\begin{eqnarray}\label{LX}
& & L( X_1,\ldots ,X_{m^\alpha})  = \prod_{j=1}^{m^{\alpha}}  f(X_{j+1} / X_{1} \ldots X_{j}) =
\nonumber \\
 & = & \prod_{j=1}^{m^{\alpha}} p^{1- \left[ {X_{j+1} - R_{j}(\theta) - F_{j} \cdot (e^{-\lambda / m^{\alpha}} -1)  \over F_{j} }  \right]^{2}}(1-p)^{ \left[ {X_{j+1} - R_{j}(\theta) - F_{j} \cdot (e^{-\lambda / m^{\alpha}} -1)  \over F_{j} }  \right]^{2}}
\end{eqnarray}

By maximizing $ L( X_1,\ldots ,X_{m^\alpha})  $ with respect to the parameter $\theta$, we obtain the following expression of the maximum likelihood estimator
\begin{small}
\begin{equation*}
\hat{\theta}_{ML}^m = m {  \sum_{j=1}^{m^{\alpha}} F^{-2}_{j} \cdot  \left( \Delta X_{j+1} - T_{j}  \right)   \left( X_{j} - S_{j-1} \right)  + \left( e^{-\lambda /m^{\alpha}}  -1  \right) \sum_{j=1}^{m^{\alpha}} F^{-1}_{j}  \cdot  \left( X_{j} - S_{j-1}  \right)   \over \sum_{j=1}^{m^{\alpha}} F^{-2}_{j} \cdot \left( X_{j} - S_{j-1}  \right)^{2}    }
\end{equation*}
\end{small}
\end{proof}

By (\ref{modeloF}) and (\ref{ml}) 
\begin{equation}\label{difml}
\hat{\theta}_{ML} -\theta = m {  \sum_{j=1}^{m^{\alpha}} \left[ F^{-1}_{j} \cdot \left( X_{j} - S_{j-1}  \right) \right] \left[ \left( e^{-\lambda /m^{\alpha}}  -1  \right)  + \eta^{m^{\alpha}}_{j+1} \right]  \over \sum_{j=1}^{m^{\alpha}} F^{-2}_{j} \left( X_{j} - S_{j-1} \right)^{2}    }
\end{equation}

We are in position to state the main results of the paper.\\
\subsection{Main Result}

\begin{thm} \label{theorem1}
Let $ \hat{\theta}^m_{LS}$ be the Least square estimator for $\theta$ in the model (\ref{modeloF}). Then: 
\begin{equation*}
\mathbb{E} \left( \hat{\theta}_{LS}^m - \theta | X_{1} = x_{1} , ..., X_{j} = x_{j} \right) = 0. 
\end{equation*}
and 
\begin{equation*}
Var \left( \hat{\theta}_{LS}^m - \theta | X_{1} = x_{1} , ..., X_{j} = x_{j} \right) \rightarrow 0, 
\end{equation*}
as $m \rightarrow \infty$. 
\end{thm}

\begin{thm} \label{theorem2}
Let $ \widehat{\theta}_{ML}^m$ be the maximum likelihood estimator for $\theta$ in the model (\ref{modeloF}). Then 
\begin{equation*}
\mathbb{E} \left( \hat{\theta}_{ML}^m - \theta | X_{1} = x_{1} , ..., X_{j} = x_{j} \right) \rightarrow 0 .
\end{equation*}
and   
\begin{equation*}
Var \left( \hat{\theta}_{ML}^m - \theta | X_{1} = x_{1} , ..., X_{j} = x_{j} \right) \rightarrow 0, 
\end{equation*}
as $m \rightarrow \infty$.
\end{thm}

\begin{rem} Theorem (\ref{theorem1}) and (\ref{theorem2}) gives the conditional consistency in mean square of the LS and MLE estimator for $\theta$.
\end{rem}

\section{Proofs: Preliminary lemmas}
In order to prove ours main results we will need the following lemmas to control the denominator $F_j^2$:

\begin{lem}
Let $M \geq 1$, and  $A_{M}$ given by 
\begin{equation}
A_{M} = \sum_{j=1}^{M-1} F^{-1}_{j} \cdot  \left( x_{j}  - S_{j-1}  \right) \eta_{j+1} ,
\end{equation}
then $A_{M}$ is $\mathscr{F}^{X}_{j} = \mathscr{F}^{\eta}_{j} $-martingale. 
\end{lem}
\begin{proof}
Let us compute $\mathbb{E} \left(A_{M+1} | \mathscr{F}^{X}_{M} \right) $, 
\begin{eqnarray}
\mathbb{E} \left(A_{M+1} | \mathscr{F}^{X}_{M} \right)  &=& \mathbb{E} \left( \sum_{j=1}^{M} F^{-1}_{j} \cdot  \left( x_{j}  - S_{j-1}  \right) \eta_{j+1}    \vert   \mathscr{F}^{X}_{M} \right) \nonumber \\ \nonumber
&=& \sum_{j=1}^{M-1} F^{-1}_{j} \cdot \left( x_{j}  - S_{j-1} \right) \eta_{j+1} + \mathbb{E} \left(  F^{-1}_{M}  \cdot \left( x_{M}  - S_{M-1}  \right) \eta_{M+1}     \vert   \mathscr{F}^{X}_{M} \right) \nonumber \\ \nonumber
&=& \sum_{j=1}^{M-1} F^{-1}_{j} \cdot \left( x_{j}  - S_{j-1}  \right) \eta_{j+1} = A_{M}.
\end{eqnarray}
\end{proof}
Moreover, $\langle A \rangle_{M} = \sum_{j=1}^{M-1} F^{-2}_{j} \cdot \left( x_{j}  - S_{j-1} \right)^{2}  \kappa_{M-1} (1- \kappa_{M-1}), $ where $\langle A \rangle_{M} $ is the bracket of the discrete martingale $(A_{M})_{M \geq 1}$and $\kappa_{M}$ is given by (\ref{kappa}).
\begin{lem}\label{cotaF}
Let $F_j$ given in   (\ref{F}). Then, for every $j \geq 0$, we have
\begin{equation}
 F^{2}_{j} \leq \lambda^{-1} m^{-2H} \ \ \  \mbox{and}  \ \ \ c_{H} m^{1-2H} \leq F^{2}_{j}. 
\end{equation}
\end{lem}
\begin{proof}
First, we use the inequality given in proposition 3.3 in \cite{ar} for the square of the increments of the fractional Poisson random walk noise, that is:
\begin{equation}\label{3.3}
\mathbb{E} \left(N_{j+1}^{m,H} -N_{j}^{m,H}  \right)^{2} \leq {\kappa_{m} (1- \kappa_{m}) \over  \lambda }m^{-2H}. 
\end{equation}
Then, by (\ref{incrementoNN})
\begin{equation}\label{incs}
 \mathbb{E} \left(N_{j+1}^{m,H} -N_{j}^{m,H}  \right)^{2}  = \mathbb{E} \left(  \sum_{i=1}^{j} f_{ij} \eta_{i}^{m} + F_{j} \eta_{j+1}^{m} \right)^{2}.
\end{equation}
If we analyze the second term in (\ref{incs}), we obtain that:
\begin{eqnarray}\label{plug}
\mathbb{E} \left(  \sum_{i=1}^{j} f_{ij} \eta_{i}^{m} + F_{j} \eta_{j+1}^{m} \right)^{2}  &=&  \sum_{i=1}^{j} f_{ij} \mathbb{E} \left[ \left(  \eta_{i}^{m} \right)^{2} \right] + F_{j} \mathbb{E}  \left[ \left( \eta_{j+1}^{m} \right)^{2} \right]  
\nonumber \\
&=&  \kappa_{m} (1- \kappa_{m})  \left[ \sum_{i=1}^{j} f^{2}_{ij}  + F^{2}_{j} \right]  , 
\end{eqnarray}
with $\kappa_{m}$ given in (\ref{kappa}). Combining equations (\ref{incs}) and (\ref{plug}), we obtain that 

\begin{equation*}
\kappa_{m} (1- \kappa_{m}) \left[ \sum_{i=1}^{j} f^{2}_{ij}   + F^{2}_{j} \right] \leq  {\kappa_{m} (1- \kappa_{m}) \over  \lambda } m^{-2H}.
\end{equation*} 

By the last inequality and since $\kappa_{m} (1- \kappa_{m}) > 0$  we have $F_{j}^{2}  \leq \lambda^{-1} m^{-2H}$. 
That complete the first part of the Lemma. 

In the case of the lower bound we use the expression of the kernel $K$ given by (\ref{Kernel}), and since $\tau \geq s$ 
and $H > 1/2$ we have (the constant $c_{H}$ may change from line to line)
\begin{eqnarray}
F_{j} &=&  c_{H} m \int_{j /m}^{j+1 \over m}  s^{\frac{1}{2}-H} \left(  \int_{s}^{j+1 \over m } \tau^{H-\frac{1}{2}} \left( \tau - s \right)^{H-\frac{3}{2}} d\tau \right) ds \nonumber  \\ \nonumber 
& \geq &  c_{H} m \int_{j /m}^{j+1 \over m}  \left(  \int_{s}^{j+1 \over m } \left( \tau - s \right)^{H-\frac{3}{2}} d\tau \right)  ds\nonumber  \\ \nonumber 
&=& c_{H} m \int_{j /m}^{j+1 \over m}    \left( {j+1 \over m} - s \right)^{H-\frac{1}{2}}    ds\nonumber  \\ \nonumber 
& = & c_{H} m^{1/2 -H}. 
\end{eqnarray}
\end{proof}
\vspace{-0.7cm}
\begin{lem}\label{lema3}
For any  $\alpha > 1$, and
\begin{equation}\label{bracket}
\langle A \rangle_{m^{\alpha}} = \sum_{j=1}^{m^{\alpha}} F^{-2}_{j} \cdot  \left( x_{j}  - S_{j-1}  \right)^{2}  \kappa_{m^{\alpha}} (1- \kappa_{m^{\alpha}})
\end{equation}
we have 
\begin{equation}
 \langle A \rangle_{m^{\alpha}}   \geq c_{H, \lambda} \cdot  \kappa_{m^{\alpha}} (1- \kappa_{m^{\alpha}})  m^{\alpha -1}.
\end{equation}
\end{lem}
\begin{proof}
Using the recurrence relation (\ref{modeloNN}) and (\ref{incrementoNN}), we have for $j \geq 1$ 
\begin{equation}
X_{j} - S_{j-1}  = g_{j} \left( X_{1},..., X_{j-1} \right) + F_{j-1} \eta_{j}^{m^{\alpha}},
\end{equation}
where $g_{j}$ is a linear function on $\left( X_{1} ,..., X_{j-1} \right)$ (see proposition 2 in \cite{Rif}). This implies 
\begin{eqnarray}\label{X-S}
\mathbb{E} \left( X_{j} - S_{j-1}  \right)^{2} &=&  \mathbb{E} \left( g_{j} \left( X_{1},..., X_{j-1} \right)  \right)^{2} \nonumber  \\ \nonumber
&+ & 2\mathbb{E} \left( F_{j-1}g_{j} \left( X_{1},..., X_{j-1} \right)  \eta_{j}^{m^{\alpha}} \right) + F_{j-1}^{2}  \kappa_{m^{\alpha}} (1- \kappa_{m^{\alpha}}) \nonumber  \\ 
&=&  \mathbb{E} \left( g_{j} \left( X_{1},..., X_{j-1} \right)  \right)^{2} + F_{j-1}^{2}  \kappa_{m^{\alpha}} (1- \kappa_{m^{\alpha}}) 
\end{eqnarray}

where the last equality is due to the fact that  
\begin{eqnarray}
\mathbb{E} \left( F_{j-1}g_{j} \left( X_{1},..., X_{j-1} \right)  \eta_{j}^{m^{\alpha}} \right)  &=&  \mathbb{E} \left( \mathbb{E} \left( F_{j-1}g_{j} \left( X_{1},..., X_{j-1} \right)  \eta_{j}^{m^{\alpha}} \right) | \mathscr{F}_{j-1} \right) \nonumber \\
&=&   \mathbb{E} \left(  F_{j-1}g_{j} \left( X_{1},..., X_{j-1} \right) \mathbb{E} \left( \eta_{j}^{m^{\alpha}}  | \mathscr{F}_{j-1} \right) \right)  \nonumber \\
&=& 0. \nonumber 
\end{eqnarray}
Therefore, by (\ref{X-S}), and for any $j \geq 1$, we obtain 
\begin{equation}\label{x-s}
\mathbb{E} \left( X_{j} - S_{j-1} \right)^{2} \geq   F_{j-1}^{2}  \kappa_{m^{\alpha}} (1- \kappa_{m^{\alpha}}) 
\end{equation}
Using (\ref{x-s}), we have that  $
\langle A \rangle_{m^{\alpha}}  \geq  \kappa_{m^{\alpha}} (1- \kappa_{m^{\alpha}})  \sum_{j=1}^{m^{\alpha}} { F^{2}_{j-1}  \over F_{j}^{2} }. $
Finally, by  \textbf{Lemma  (\ref{cotaF} }), we have 
\begin{eqnarray}\label{cotaA}
   \langle A \rangle_{m^{\alpha}}   \geq   \kappa_{m^{\alpha}} (1- \kappa_{m^{\alpha}})  \sum_{j=1}^{m^{\alpha}} { F^{2}_{j-1}  \over F_{j}^{2} } \geq c_{H, \lambda} \cdot  \kappa_{m^{\alpha}} (1- \kappa_{m^{\alpha}})  m^{\alpha -1}
\end{eqnarray} 
\end{proof}
\vspace{-0.7cm}
\subsection{Proof of Theorem \ref{theorem1}}
We point  out here, that a general result concerning the asymptotic behaviour of (\ref{ls}) seems difficult to be obtained given the  correlation structure of the random variables $X_{j}$. Then, we will be able to prove, as we said before a conditional result 

\begin{proof}
The first part of the proof is rather simple since
\begin{small}
\begin{equation}
\mathbb{E} \left( \hat{\theta}_{LS} ^m- \theta | X_{1} = x_{1} , ..., X_{j} = x_{j} \right)=  {m \kappa_{m^{\alpha}} (1- \kappa_{m^{\alpha}}) \over  \langle A \rangle_{m^{\alpha}}   } \mathbb{E} \sum_{j=1}^{m^{\alpha}} { \left( X_{j}  - S_{j-1} \right) \over F^{-1}_{j} } \eta_{j+1}^{m^{\alpha}} = 0  \nonumber 
\end{equation}
\end{small} 
Now, for the  variance we have,
\begin{small}
\begin{eqnarray*}
Var \left(\hat{\theta}_{LS}^m - \theta | X_{1} = x_{1} , ..., X_{j} = x_{j} \right)  &=& \mathbb{E} \left( \left( \hat{\theta}_{LS}^m - \theta \right)^{2} | X_{1} = x_{1} , ..., X_{j} = x_{j} \right)  \\ 
&=& {m^{2} \kappa_{m^{\alpha}}^{2} \left( 1 - \kappa_{m^{\alpha}} \right)^{2}  \over  \langle A \rangle_{m^{\alpha}}^{2} } \mathbb{E}  \left(  \sum_{j=1}^{m^{\alpha}} F^{-1}_{j}  \left( X_{j}  - S_{j-1}  \right) \eta_{j+1}^{m^{\alpha}}  \right)^{2}  \\ 
&=& {m^{2} \kappa_{m^{\alpha}}^{2} \left( 1 - \kappa_{m^{\alpha}} \right)^{2}  \over  \langle A \rangle_{m^{\alpha}}^{2} }   \sum_{j=1}^{m^{\alpha}} F^{-2}_{j}  \left( X_{j}  - S_{j-1}  \right)^2 \mathbb{E}  (\eta_{j+1}^{m^{\alpha}})^2  \\ 
&=& {m^{2} \kappa_{m^{\alpha}}^{2} \left( 1 - \kappa_{m^{\alpha}} \right)^{2} \over  \langle A \rangle_{m^{\alpha}} }  \\ 
& \leq & c_{H,\lambda} m^{3 -\alpha} \kappa_{m^{\alpha}} \left( 1 - \kappa_{m^{\alpha}} \right)  = m^{3 -\alpha} e^{-\lambda \over m^{\alpha}} \left( 1 - e^{-\lambda \over m^{\alpha}} \right) 
\end{eqnarray*}

Clearly this goes to 0 as $m \rightarrow \infty$ for $\alpha >3/2$. The first equality is due to \ref{difls}. Then, we use independence on the $\eta_{j+1}^{m^{\alpha}}$ random variables and \ref{bracket}. Finally, the last inequality results by lemma \ref{lema3}. 
\end{small}

\end{proof}

\subsection{Proof of Theorem \ref{theorem2}}
In this section we will prove our second result, and as before a general result concerning the asymptotic behaviour of (\ref{ml}) seems difficult to be obtained given the correlation structures of the random variables $X_{j}$. 
By (\ref{difml}) we have that 
\begin{eqnarray*}
& &\mathbb{E} \left(  \hat{\theta}_{ML}^m - \theta | X_{1} =x_{1}, ..., X_{j} = x_{j} \right)  \nonumber \\
&=& {m \over \sum_{j=1}^{m^{\alpha}} F^{-2}_{j} \cdot \left( X_{j} - S_{j-1} \right)^{2} } \sum_{j=1}^{m^{\alpha}}  \left[  \left( x_{j} - S_{j-1}  \right)  \over F_{j}\right] \mathbb{E} \left[ \left( e^{-\lambda /m^{\alpha}}  -1  \right)  + \eta^{m^{\alpha}}_{j+1} \right]  \\ 
&=& {m  \left(  e^{-\lambda /m^{\alpha}}  -1  \right) \over \sum_{j=1}^{m^{\alpha}} F^{-2}_{j} \cdot \left( X_{j} - S_{j-1} \right)^{2}   } \sum_{j=1}^{m^{\alpha}}  \left[  \left( x_{j} - S_{j-1}  \right)  \over F_{j}\right]
\\
& \leq & {m^{1+\alpha /2 }  \left(  e^{-\lambda /m^{\alpha}}  -1  \right) \over \sum_{j=1}^{m^{\alpha}} F^{-2}_{j} \cdot \left( X_{j} - S_{j-1} \right)^{2}   }  \sqrt{ \sum_{j=1}^{m^{\alpha}}  \left[  \left( x_{j} - S_{j-1} \right)  \over F_{j}\right]^{2} }   \\  
& = & {m^{1+\alpha /2}  \left(  e^{-\lambda /m^{\alpha}}  -1  \right) \sqrt{\kappa_{m^{\alpha}} \left( 1- \kappa_{m^{\alpha}} \right)} \over \sqrt{ \langle A \rangle_{m^{\alpha}} }  }  \\ 
& \leq & C_{H, \lambda} m^{3/2}  \left(  e^{-\lambda /m^{\alpha}}  -1  \right) 
 \end{eqnarray*}
And this goes to $0$ as $m \rightarrow \infty$, for  $\alpha > 3/2$ . Concerning the conditional variance we have 
\begin{eqnarray}
Var \left(  \hat{\theta}_{ML} - \theta | X_{1} =x_{1}, ..., X_{j} = x_{j} \right) &=&  \mathbb{E} \left(  (\hat{\theta}_{ML} - \theta)^{2} | X_{1} =x_{1}, ..., X_{j} = x_{j} \right) \nonumber \\ \nonumber 
& + & \left(\mathbb{E} \left(  \hat{\theta}_{ML} - \theta | X_{1} =x_{1}, ..., X_{j} = x_{j} \right) \right)^{2}
\nonumber \\ \nonumber 
& \leq &  I_{1} + I_{2}.
\end{eqnarray}

First for $I_{1}$ if we define $A^{*}_{m^{\alpha}} = \sum_{j=1}^{m^{\alpha}} F^{-2}_{j} \cdot \left( X_{j} - S_{j-1}  \right)^{2}$, we can note that $\langle  A \rangle_{m^{\alpha}} = A^{*}_{m^{\alpha}} \kappa_{m^{\alpha}} (1- \kappa_{m^{\alpha}})$. Now, by (\ref{cotaA}) we have 

\begin{eqnarray*}
I_{1} &=&\mathbb{E} \left(  (\hat{\theta}_{ML} - \theta)^{2} | X_{1} =x_{1}, ..., X_{j} = x_{j} \right)  \\
&=& \mathbb{E} \left(  \left[ m {  \sum_{j=1}^{m^{\alpha}} \left[ F^{-1}_{j} \cdot \left( X_{j} - S_{j-1}  \right) \right] \left[ \left( e^{-\lambda /m^{\alpha}}  -1  \right)  + \eta^{m^{\alpha}}_{j+1} \right]  \over \sum_{j=1}^{m^{\alpha}} F^{-2}_{j} \left( X_{j} - S_{j-1} \right)^{2}    } \right]^{2} \vert X_{1} =x_{1}, ..., X_{j} = x_{j}  \right) \\ 
& =  & {m^{2} \over \left( A^{*}_{m^{\alpha}} \right)^{2} }\mathbb{E} \left[ \left( \left( e^{-\lambda /m^{\alpha}}  -1 \right)\sum_{j=1}^{m^{\alpha}}  \left[  \left( x_{j} - S_{j-1}  \right)  \over F_{j}\right]        + \sum_{j=1}^{m^{\alpha}}  \left[  \left( x_{j} - S_{j-1} \right)  \over F_{j}\right]  \eta^{m}_{j+1} \right)^{2} \right]
\\ 
& \leq & {2m^{2} \over \left( A^{*}_{m^{\alpha}} \right)^{2}} \left( e^{-\lambda /m^{\alpha}}  -1 \right)^{2} \left( \sum_{j=1}^{m^{\alpha}}  \left[  \left( x_{j} - S_{j-1}  \right)  \over F_{j}\right] \right)^{2} +  {2m^{2} \over \left( A^{*} \right)^{2}_{m^{\alpha}}} \sum_{j=1}^{m^{\alpha}}  \left[  \left( x_{j} - S_{j-1}  \right)  \over F_{j}\right]^{2} \mathbb{E} \left[ \eta^{m^{\alpha}}_{j+1} \right]^{2}  
\\ 
& \leq &  {2m^{2 + \alpha} \over  A^{*}_{m^{\alpha}}}  \left( e^{-\lambda /m^{\alpha}}  -1 \right)^{2}    + {2m^{2} \over A^{*}_{m^{\alpha}}} \kappa_{m^{\alpha}} \left( 1 - \kappa_{m^{\alpha}}  \right)  \\ 
& \leq &  C_{H, \lambda} m^{3} \left( e^{-\lambda /m^{\alpha}}  -1 \right)^{2}    + m^{3-\alpha} e^{-\lambda \over m^{\alpha}} \left( 1 - e^{-\lambda \over m^{\alpha}} \right) \\ 
\end{eqnarray*}

Next, for $I_{2}$ we get as a direct consequence from the first part of theorem (\ref{theorem2})) that 
\begin{eqnarray*}
I_{2} & \leq & C_{H, \lambda} m^{3}  \left(  e^{-\lambda /m^{\alpha}}  -1  \right)^{2}. 
\end{eqnarray*}
So finally
\begin{small}
\begin{equation*}
Var \left(  \hat{\theta}_{ML} - \theta | X_{1} =x_{1}, ..., X_{j} = x_{j} \right) \leq C_{H, \lambda} m^{3}  \left(  e^{-\lambda /m^{\alpha}}  -1  \right)^{2} + m^{3-\alpha} e^{-\lambda \over m^{\alpha}} \left( 1 - e^{-\lambda \over m^{\alpha}} \right), 
\end{equation*}
\end{small} 
and this goes to 0 as $m \rightarrow \infty$ for $\alpha >3/2$.

\begin{rem}
From the results of Theorem (\ref{theorem1}) and Theorem (\ref{theorem2}), we need to take $\alpha > 3/2$, nevertheless if we take $\lambda = m \ln (2) $, we recover the results from the fractional Brownian motion  case (see \cite{ber1}, \cite{ber2} and \cite{Rif}) in this case  $\alpha >1$ becomes the necessary condition. 
\end{rem}

\section{Simulation Study}

Having described the LSE and MLE estimators  for the fractional  O-U  Poisson process (\ref{fou}),  
we now proceed to evaluate the performance of those estimators and  show the shape of the  limit distribution.

So that, we simulate the discrete process $N_{t}^{m,H}$ for   $m=10$ and $m=100$. For each case, we calculate $100$ estimations. In  both cases, we use $\alpha = 2$  and different values of $H$. Then, we compute the estimators given in (\ref{ls}) and (\ref{ml}). Recall, that we have explicit forms for ours estimators. Our simulation shows that the behaviour of the limit distribution is close to a Gaussian one when $H \leq 3/4$. Also, in the case of  $H>3/4$, the behaviour of the limit distribution is close to a non-symmetric one.

We give in the following tables the mean and the variance of these estimations for m=10 and m=100. The results are the followings:
\begin{small}
\begin{table}[h!]
\centering
\begin{tabular}{c|cc|cc}
\hline
$\theta$ & $\hat{\theta}_{LSE}^m$ & $Var\left( \hat{\theta}_{LSE}^m \right)$ & $\hat{\theta}_{MLE}^m$ & $Var\left( \hat{\theta}_{MLE}^m \right)$ \\ \hline
0.1 & 0.07421 & 0.6455 & 0.16123 & 0.1454\\
0.5 & 0.58631 & 0.2431 & 0.47112 & 0.1254\\
0.9 & 0.99112 & 0.1632 & 0.87682 & 0.1023\\
\hline
\end{tabular}

\vspace{0.1cm}

\centering
\begin{tabular}{c|cc|cc}
\hline
$\theta$ & $\hat{\theta}_{LSE}^m$ & $Var\left( \hat{\theta}_{LSE} ^m\right)$ & $\hat{\theta}_{MLE}^m$ & $Var\left( \hat{\theta}_{MLE}^m \right)$ \\ \hline
0.1 & 0.05432 & 0.0382 & 0.12442 & 0.1223\\
0.5 & 0.46508 & 0.0126 & 0.50078 & 0.0883\\
0.9 & 0.91196 & 0.0643 & 0.90532 & 0.1032\\
\hline
\end{tabular}

\vspace{0.1cm}

\centering
\begin{tabular}{c|cc|cc}
\hline
$\theta$ & $\hat{\theta}_{LSE}^m$ & $Var\left( \hat{\theta}_{LSE}^m \right)$ & $\hat{\theta}_{MLE}^m$ & $Var\left( \hat{\theta}_{MLE}^m \right)$ \\ \hline
0.1 & 0.05923 & 0.00432 & 0.08980 & 0.01021\\
0.5 & 0.48201 & 0.00636 & 0.49324 & 0.04543\\
0.9 & 0.91324 & 0.01324 & 0.90414 & 0.00892\\
\hline
\end{tabular}
\caption{Results for $m=10$, $H=0.55$, $H=0.75$ and $H=0.90$, respectively}
\label{Table1}
\end{table}

\end{small}

\begin{small}

\begin{table}[h!]
\centering
\begin{tabular}{c|cc|cc}
\hline
$\theta$ & $\hat{\theta}_{LSE}^m$ & $Var\left( \hat{\theta}_{LSE}^m \right)$ & $\hat{\theta}_{MLE}^m$ & $Var\left( \hat{\theta}_{MLE}^m \right)$ \\ \hline
0.1 & 0.06366 & 0.5859 & 0.14293 & 0.1187\\
0.5 & 0.56011 & 0.2139 & 0.48084 & 0.0840\\
0.9 & 0.97218 & 0.1349 & 0.89291 & 0.0889\\
\hline
\end{tabular}

\vspace{0.1cm}

\centering
\begin{tabular}{c|cc|cc}
\hline
$\theta$ & $\hat{\theta}_{LSE}^m$ & $Var\left( \hat{\theta}_{LSE} ^m\right)$ & $\hat{\theta}_{MLE}^m$ & $Var\left( \hat{\theta}_{MLE}^m \right)$ \\ \hline
0.1 & 0.05084 & 0.0282 & 0.10782 & 0.1009\\
0.5 & 0.49508 & 0.0126 & 0.50078 & 0.0883\\
0.9 & 0.90962 & 0.0086 & 0.90968 & 0.0968\\
\hline
\end{tabular}

\vspace{0.1cm}

\centering
\begin{tabular}{c|cc|cc}
\hline
$\theta$ & $\hat{\theta}_{LSE}^m$ & $Var\left( \hat{\theta}_{LSE}^m \right)$ & $\hat{\theta}_{MLE}^m$ & $Var\left( \hat{\theta}_{MLE}^m \right)$ \\ \hline
0.1 & 0.06051 & 0.00127 & 0.09009 & 0.00948\\
0.5 & 0.49703 & 0.00036 & 0.49998 & 0.00917\\
0.9 & 0.90903 & 0.00033 & 0.90874 & 0.00619\\
\hline
\end{tabular}
\caption{Results for $m=100$, $H=0.55$, $H=0.75$ and $H=0.90$, respectively}
\label{Table2}
\end{table}

\end{small}

\newpage
\begin{figure}[h!]
	\centering
	\includegraphics[width=4in, height= 3.5in]{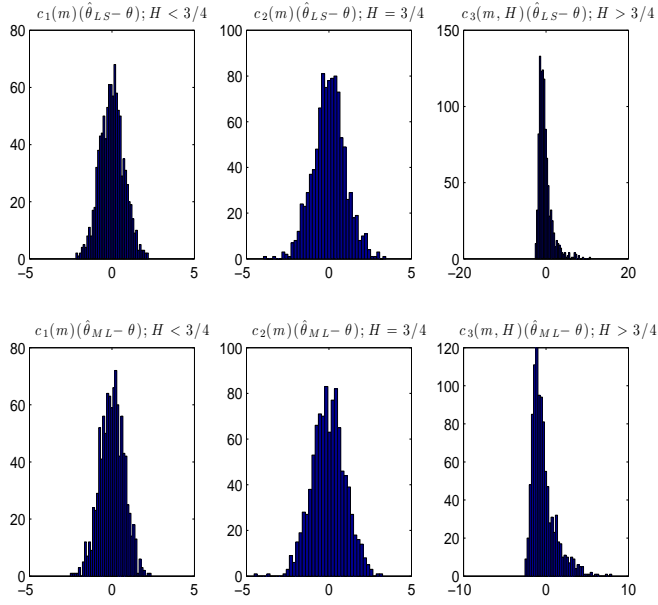}
	\caption[hist1]
	{Histogram of normalized estimators for different values of $H$  and  $\theta =0.4$. }
	\label{hist1}
\end{figure}
Figure \ref{hist1} shows the histograms of normalized estimators, $\hat{\theta}_{LSE}^m$ and $\hat{\theta}_{MLE}^m$ for different values of $H$, with $m =100$, $\alpha=2$ and  $\theta= 0.55$. The respective normalizations are $c_{1}(m) = \sqrt{m^{\alpha}}$ , $c_{1}(m) = \sqrt{m^{\alpha} \over log(m^{\alpha} )}$ and $c_{1}(m,H) = m^{\alpha (1-H)}$. We can draw a conclusion that have asymptotic normalities for the cases $H \leq 3/4$ and asymmetric behaviour for $H>3/4$, although we cannot prove the asymptotic normality of $\hat{\theta}_{LSE}$ and $\hat{\theta}_{MLE}$ in theory.

\begin{table}[h!]
\centering
\begin{tabular}{c|cc|cc|cc}
\hline
$\theta$ & $\hat{\theta}_{MLE}^m$ & $Var(\hat{\theta}_{MLE}^m)$ & $\hat{\theta}_{MLE}^m$ & $Var(\hat{\theta}_{MLE}^m)$ & $\hat{\theta}_{MLE}^m$ & $Var(\hat{\theta}_{MLE}^m)$ \\ \hline
0.1 & 0.15312 & 0.18321 & 0.14213 & 0.12114 & 0.13569 & 0.11324\\
0.5 & 0.46100 & 0.25578 & 0.52990 & 0.14432 & 0.52878 & 0.01654\\
0.9 & 0.93476 & 0.15787 & 0.91891 & 0.10345 & 0.91003 & 0.00981\\
\hline
\end{tabular}
\caption{Results for $m=10$, $H=0.55$ , $H=0.75$ and $H=0.90$, from left to right, respectively.}
\label{Table3}
\end{table}

\newpage

\begin{table}[h!]
\centering
\begin{tabular}{c|cc|cc|cc}
\hline
$\theta$ & $\hat{\theta}_{MLE}^m$ & $Var(\hat{\theta}_{MLE}^m)$ & $\hat{\theta}_{MLE}^m$ & $Var(\hat{\theta}_{MLE}^m)$ & $\hat{\theta}_{MLE}^m$ & $Var(\hat{\theta}_{MLE}^m)$ \\ \hline
0.1 & 0.13455 & 0.15859 & 0.12293 & 0.11875 & 0.12032 & 0.09121\\
0.5 & 0.47987 & 0.22341 & 0.51876 & 0.08402 & 0.51434 & 0.01135\\
0.9 & 0.92254 & 0.14327 & 0.91004 & 0.08891 & 0.90023 & 0.00513\\
\hline
\end{tabular}
\caption{Results for $m=100$, $H=0.55$ , $H=0.75$ and $H=0.90$, from left to right, respectively.}
\label{Table4}
\end{table}

Table \ref{Table3} and \ref{Table4} shows the results of the estimation in the case $\lambda = m log(2)$ for $m=10$ and $m=100$. Here the Maximum Likelihood case is indicated, since as we said before in the Gaussian context the MLE and LSE coincide.

\begin{figure}[h!]
	\centering
	\includegraphics[width=5in , height= 1.7in]{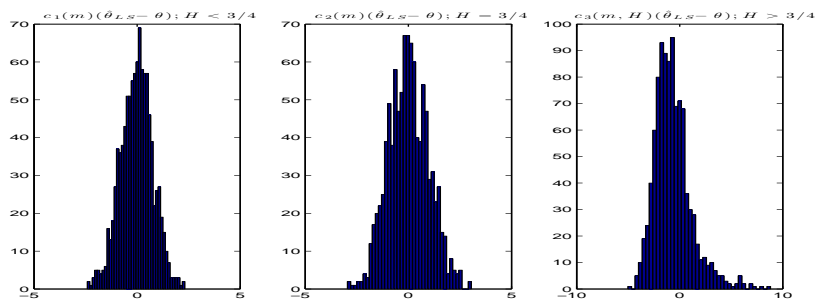}
	\caption[hist2]
	{Histogram of normalized estimators for different values of $H$  and  $\theta =0.4$. }
	\label{hist2}	
\end{figure}

Figure \ref{hist2} shows the histograms of normalized estimator $\hat{\theta}_{MLE}^m$ in the case $\lambda= m \ln(2)$ for different values of $H$, with $m =100$ , $\alpha =2 $ and  $\theta= 0.55$, the normalization constants are the same as in the general case. 

\newpage

\begin{figure}[h!]
	\centering
	\includegraphics[width=3.5in , height= 2.5in ]{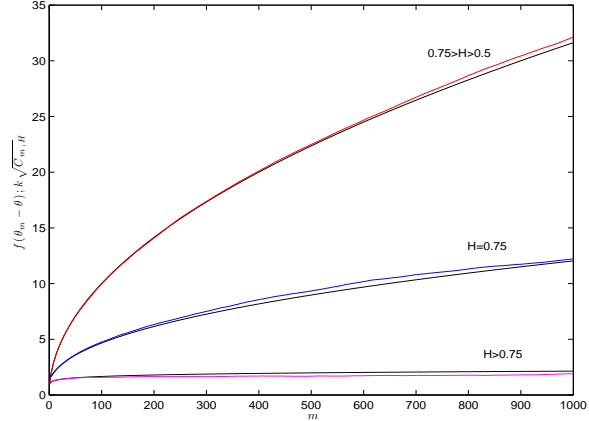}
	\caption[funcion]
	{Weak convergence: Empirical rates vs theorical rates.}
	\label{rates}
\end{figure}
In figure \ref{rates}, comparisons of empirical rates versus theoretical ones are shown, that is, we compare the empirical variance versus the theoretical by means of the Monte Carlo method. It can be see that apparently the chosen rates seem appropriate.
%
%


From the simulation results presented in Tables \ref{Table1}, \ref{Table2} , \ref{Table3} and \ref{Table4}, and for the shape of the asymptotic distribution for the  parameter $\theta$ (Figure \ref{hist1}), the following conclusions may be summarized

\begin{enumerate}
\item  The MSE and bias values of the proposed LSE and MLE decrease when sample sizes increase. Thus, we show that the proposed LSE and MLE provides consistent estimates. According to the MSE criterion, the proposed MLE apparently shows better performance than the LSE for all considered values of $H$ when sample sizes are $m = 10$ and $100$. Since the MLE is asymptotically the best, it provides the best
performance for $H=0.5$, as expected. 

\item For all situations studied here, the parameter estimation improve as the value of $H$ approaches 1.

\item Based on the results presented in Figure 2,  the shape of the asymptotic distribution seems like a Gaussian one for $H \le 3/4$ and  a non-symmetric one  for  $H>3/4$. 

\end{enumerate}

To sum up, the proposed LSE and MLE  for the  parameter $\theta$ in the O-U fractional Poisson process 
provides a good performance in  the most of the considered samples and $H$ values.  

\vspace{0.5cm}

{\bf Acknowledgments:}
The authors thank the Editor, an Associate Editor and two anonymous reviewers for valuable comments that led to great improvement of this paper. This  research was partially supported by Project ECOS - CONICYT C15E05, REDES 150038 and MATHAMSUD 16-MATH-03 SIDRE Project. H\'ector Araya was partially supported by Beca CONICYT-PCHA / Doctorado Nacional / 2016-21160138, Natalia Bahamonde was partially supported by FONDECYT Grant 1160527, Soledad Torres was partially supported by FONDECYT Grant 1171335.

\bibliographystyle{plain} 
\bibliography{HA-NB-TR-ST.bib}
\end{document}